\documentclass[12pt]{amsart}

\usepackage{amsmath,amsthm,amsfonts,amssymb,mathrsfs}
\usepackage[margin=1in]{geometry}
\usepackage[hyperfootnotes=false]{hyperref}
\usepackage{verbatim}
\usepackage{todonotes}
\usepackage{fancyvrb }
\usepackage{color}
\newcommand{\Z}{\mathbb{Z}}

\newtheorem{theorem}{Theorem}[section]

\newtheorem{proposition}[theorem]{Proposition}

\newtheorem{conjecture}[theorem]{Conjecture}

\newtheorem{corollary}[theorem]{Corollary}

\newtheorem{lemma}[theorem]{Lemma}
\input xy
\xyoption{all}

\theoremstyle{definition}

\newtheorem{remark}[theorem]{Remark}

\DeclareMathOperator{\opH}{H}
\DeclareMathOperator{\ext}{Ext}

\begin{document}

\title[Restricted partitions and $SL_2$-cohomology]{Restricted partitions and $SL_2$-cohomology}

\author{Steven Benzel}
\address{Department of Mathematics\\ University of North Georgia--Gainesville \\ Oakwood, GA~30566, USA}
\email{Steven.Benzel@ung.edu}

\author{Scott Conner}
\address{Department of Mathematics\\ University of North Georgia--Gainesville \\ Oakwood, GA~30566, USA}
\email{SKCONN5042@ung.edu}

\author{Nham Ngo}
\address{Department of Mathematics\\ University of North Georgia--Gainesville \\ Oakwood, GA~30566, USA}
\email{nvngo@ung.edu}

\author{Khang Pham}
\address{Department of Mathematics\\ University of North Georgia--Gainesville \\ Oakwood, GA~30566, USA}
\email{KVPHAM6466@ung.edu}

\date{\today}

\maketitle

\begin{abstract}
The aim of this paper is twofold. First, we study the number of partitions of a positive integer $m$ into at most $n$ parts in a given set $A$. We prove that such a number is bounded by the $n$-th Fibonacci number $F(n)$ for any $m$ and some family of sets $A$ including sets of powers of an integer. Then, in the second part of the paper, we provide new results in bounding the cohomology of the simple algebraic group $SL_2$ with coefficients in Weyl modules. 
\end{abstract}

\section{Introduction}
Let $A$ be a subset of $\Z^+$ and $m$ in $\Z^+$. A {\it restricted partition} of $m$ with parts in $A$ is a decomposition
\begin{align}\label{partitionform}
 m=\alpha_1+\alpha_2+\cdots+\alpha_t 
\end{align}
where $\alpha_i$s are not necessarily distinct elements in $A$ and $\alpha_1\le\alpha_2\le\cdots\le\alpha_t$. Each $\alpha_i$ is called a {\it part} of the partition, and $t$ is the number of parts in the partition. For example, let $A_2=\{2^i:i\in\mathbb N \}$ the set of all powers of 2. Then $5$ has 4 restricted partitions in $A_2$: $1+1+1+1+1, 1+1+1+2, 1+2+2,$ and $1+4.$ These are called $2$-ary partitions of $5$. In this paper, we are interested in restricted partitions of the form \eqref{partitionform} where $t$ is at most $n$ for some given positive integer $n$. We call them partitions of $m$ into at most $n$ parts from $A$ and denote $P(m,n,A)$ the number of such partitions.
Note that when $A=\Z^+$, i.e., there is no restriction for parts of partitions, computing $P(m,n,A)$ is a classical problem. Some first results on small values of $n$ date back to the 19$^{th}$ century by Herschel, Cayley, and Sylvester \cite{WW}. Recently, some progress has been made by \cite{K},\cite{M}, and \cite{O}. In terms of an arbitrary set $A$, not much is known for $P(m,n,A)$. One may refer to \cite{CW} and \cite{A} for a discussion on the number of some related partitions. In Section \ref{partition} we show that for any set $A$ satisfying a mild growth condition, $P(m,n,A)$ is bounded from above by the $n$-th Fibonacci number $F(n)$ for all $m,n\in\Z^+$. This bound is universal in the sense that it does not depend on the integer $m$. As a consequence, our result deduces an upper bound for the number of $q$-ary partitions of $m$ into at most $n$ parts for any positive integer $q \ge 2$. 

Our work on restricted partitions is motivated by studying the cohomology of the simple algebraic group $SL_2$ defined over an algebraically closed field $k$ of prime characteristic $p$. Note that, in general, the cohomology of algebraic groups is widely unknown and only a few cases have been computed explicitly, see \cite{Jan} for a summary of this theory. Although the case of $SL_2$ has been extensively studied, there are still open problems. For example, determining a closed formula of the dimension of the cohomology for a simple (or indecomposable) module remains unknown. In fact, it is already a challenging task to find a sharp upper bound for this number. For the last few years, the third author has been interested in bounding the dimension of $\opH^n(SL_2,V(m))$, the $SL_2$-cohomology for the Weyl module $V(m)$ of highest weight $m\in\mathbb N$. This problem is of independent interest and has its own history \cite{EHP}, \cite{LNZ}. Together with Lux and Zhang, the third author was able to identify $\dim\opH^n(SL_2,V(m))$ with the number of solutions to a system of linear equations, hence attaining a rough upper bound, see \cite[Section 4]{LNZ} for details. In Section \ref{cohomology} we use results from Section \ref{partition} to show that for any $m,n\in\mathbb N$ 
\[ \dim\opH^n(SL_2,V(m))\le
\begin{cases}
F(n+1)~~~&\text{if}~~~p\ge 5,\\
F(2n) &\text{if}~~~p=2, 3
\end{cases}
 \]
which significantly improves the bounds in \cite[Cor. 4.3, Prop. 4.4, and Thm. 4.6]{LNZ}. 

It is worth noting that there is a desire of finding explicit universal bounds (only depending on the degree $n$) for any simple algebraic group, see \cite{B et al.} for a survey on this open problem. Currently, we are not able to generalize our results (for $SL_2$) to arbitrary algebraic groups. However, we suspect that there might be some connection between the dimension of cohomology of an algebraic group and the number of restricted vector partitions\footnote{A vector partition is a way of writing a vector with nonnegative integer entries as a sum of other vectors (with nonnegative integer entries) where the order of summands does not matter.} (a generalization of restricted integer partitions). Hence, it is reasonable to ask whether there exists a universal upper bound for the latter. This would be an interesting problem for future research.

\section{Restricted partitions of $m$ into at most $n$ parts}\label{partition}

In this section, we study the number $P(m,n,A)$ for various sets $A$. We aim to bound this number using the Fibonacci numbers. For the sake of our calculations, we first identify these partitions with integer solutions of a certain system of equations. In particular, 
we write $A=\{a_1,a_2,\ldots \}$ with $0<a_1<a_2<\cdots$. Let $m\in\mathbb Z^+$ and $a_r$ the largest number in $A$ that is no more than $m$. Then each restricted partition of the form \eqref{partitionform} can be rewritten as
\[ m=x_1a_1+x_2a_2+\cdots+x_ra_r  \]
where $x_i$s are non-negative integers. Each $x_i$ is called the {\it multiplicity} of $a_i$ in the partition. As multiplicities are allowed to be zeros, every partition of $m$ is uniquely determined by a sequence $\{x_i\}$ satisfying 
\[ m=\sum_{i=1}^\infty x_ia_i \]
where $x_i=0$ for large enough $i$ (such that $a_i>m$). Thus, the number of restricted partitions of $m$ is equal to the number of such sequences $\{x_i\}$. Now, if we require the number of parts in every partition of $m$ to be at most $n$, then $P(m,n,A)$ is in fact equal to the number of sequences $\{x_i\}$ of non-negative integers such that 
\[
\begin{cases}
\displaystyle{\sum_{i=1}^\infty x_ia_i=m,}\\
\displaystyle{\sum_{i=1}^\infty x_i\le n}.
\end{cases}
\]
Solutions to the last system are essentially the same as tuples $(x_1,x_2,\ldots,x_r)$ satisfying
\begin{align}\label{gen system}
\begin{cases}
\displaystyle{x_1a_1+x_2a_2+\cdots+x_ra_r=m,}\\
\displaystyle{x_1+x_2+\cdots+x_r\le n}
\end{cases}
\end{align}
for sufficiently large $r$. From now on, we identify $P(m,n,A)$ with the number of solutions to the system \eqref{gen system}.

We next recall the definition of Fibonacci sequence. Let 
\[F(1)=1,\quad F(2)=1, \quad  F(n)=F(n-1)+F(n-2), \]
for $n\ge 2$. We also assume $F(i)=0$ for all integers $i\le 0$. Our inductive proofs will use the following special property of the Fibonacci sequence.

\begin{lemma}\label{lem}
For each positive integer $n$, we have
\[  \sum\limits_{i=0}^nF(2i+1)= F(2n+2)\quad\quad~~~~~~~~\text{and}~~~~~~~~~\quad\quad\sum\limits_{i=0}^nF(2i)= F(2n+1)-1. \]
Consequently, we always have $$\displaystyle \sum\limits_{i=0}^{\lfloor \frac{n}{2}\rfloor}F(n-2i)\le F(n+1).$$
\end{lemma}

Proof of the lemma is straightforward. We now prove the following 

\begin{theorem}\label{restricedpar}
Let $A=\{a_1,a_2,\ldots\}\subseteq\mathbb Z^+$ where all $a_i$s are written in an increasing order and satisfy
\begin{align}\label{condition*}
 2a_{s-1}+4a_{s-2}+6a_{s-3}+\cdots+2(s-1)a_1<a_s
\end{align}
for all $s\ge 2$. Then $P(m,n,A)\le F(n)$ for all positive integers $m, n$.
\end{theorem}

\begin{proof}
We prove by induction on $n$. Obviously, $P(m,1,A)=0$ or $1$, so that it is true for $n=1$. Let's take a look further to the case $n=2$ \footnote{Since $F(1)=F(2)=1$, it is necessary to show the base case with $n=1, 2,$ for otherwise we would have no idea whether $P(m,1,A)$ is bounded by $F(1)$ or $F(2)$. We will need to do the same for other inductive proofs in this paper.}. We also have $P(m,2,A)\le 1$ because for each $m$ and $a_s$, the largest number in $A$ such that $a_s\le m$, the condition \eqref{condition*} guarantees that $a_s$ is not equal to $2a_i$ nor $a_i+a_j$ for any $1\le i,j\le s-1$. Assume that $P(m,t,A)\le F(t)$ for any $t<n$ with $n\ge 3$. Now fix $m$ and again let $a_s$ be the largest number in $A$ such that $a_s\le m$. Since $2a_{s-1}+4a_{s-2}+6a_{s-3}+\cdots+2(s-1)a_1< a_s\le m$, we must have either one (or more) of the following
\[ x_{s}\ge 1, x_{s-1}\ge 3, x_{s-2}\ge 5, \ldots, x_{1}\ge 2s-1.  \]
Solutions of \eqref{gen system} combined with $x_{s}\ge 1$ are essentially solutions of 
\begin{align*}
\begin{cases}
\displaystyle{x_1a_1+x_2a_2+\cdots+x_{s-1}a_{s-1}+x'_sa_s=m-a_{s},}\\
\displaystyle{x_1+x_2+\cdots+x_{s-1}+x'_s\le n-1}\\
\end{cases}
\end{align*}
by setting $x_{s}=1+x'_{s}$. The number of such solutions is $P(m-a_{s},n-1,A)$. Similarly, numbers of solutions to \eqref{gen system} with other conditions are respectively $P(m-3a_{s-1},n-3,A),\ldots,$ and $P(m-(2s-1)a_{1},n+1-2s,A)$. Hence, we obtain
\[ P(m,n,A)\le P(m-a_{s},n-1,A)+P(m-3a_{s-1},n-3,A)+\cdots+P(m-(2s-1)a_{1},n+1-2s,A). \]
Next, applying our inductive hypothesis and Lemma \ref{lem} we have
\[ P(m,n,A)\le F(n-1)+F(n-3)+\cdots+F(n+1-2s)\le F(n), \]
hence completing our inductive proof. 
\end{proof}

We present some examples of the set $A$ in the above theorem. 

\begin{theorem}\label{A_q>3}
Let $q$ be an integer greater than 3 and set $A_q=\{q^i:i\in\mathbb N\}$. Then $A_q$ satisfies the condition \eqref{condition*}. Consequently,
\[ P(m,n,A_q)\le F(n). \]
\end{theorem}

\begin{proof}
It suffices to show that for any $s\ge 1$
\[ 2q^{s-1}+4q^{s-2}+\cdots+ 2(s-1)q+2s<q^s. \]
This can be proven by induction on $s$. Indeed, it is easy to see that it's true for $s=1$. Now to show that  
\[ 2q^{s}+4q^{s-1}+\cdots+ 2sq+2(s+1)<q^{s+1},  \]
note that
\begin{align*}
2q^{s}+&4q^{s-1}+\cdots+ 2sq+2(s+1)\\
&=2q^s + (2q^{s-1} + 4q^{s-2} + \cdots + 2(s - 1)q + 2s) + (2q^{s-1} + 2q^{s-2} + \cdots + 2q + 2)\\
&<2q^s + q^s + 2(q^{s-1} + q^{s-2} + \cdots + q + 1)~~~\text{by induction}\\
&\le 3q^s+q^s-1<4q^s\le q^{s+1}~~~\text{using}~ q > 3~ \text{at multiple points}.
\end{align*}
This proves our induction proof. The remainder follows immediately from Theorem \ref{restricedpar}.
\end{proof}

\begin{remark}\label{counter-example for p=2}
The sets $A_2$ and $A_3$ do not satisfy \eqref{condition*}. Moreover, the inequality in the above theorem doesn't hold for $A_2$ as we have
\[ 4=1\cdot 2^2=2\cdot 2=2\cdot 1+1\cdot 2.  \]
Hence, $P(4,3,A_2)=3>F(3)$. 
\end{remark}

We next modify the condition on the set $A$ in the Theorem \ref{restricedpar} so that $A_2$ and $A_3$ will satisfy it. 

\begin{theorem}\label{A23}
Let $A=\{a_1,a_2,\ldots\}\subseteq\mathbb Z^+$ where all $a_i$s are written in an increasing order and satisfy
\begin{align}\label{condition23}
a_{s-1}+2a_{s-2}+3a_{s-3}+4a_{s-4}+\cdots+(s-1)a_1<a_{s+1}
\end{align}
for all $s\ge 2$. Then $P(m,n,A)\le F(2n-1)$ for all positive integers $m, n$.
\end{theorem}

\begin{proof}
It is straightforward to check that it is true for $n=1, 2$. Assume that $P(m,t,A)\le F(2t-1)$ for any $t<n$ with $n\ge 3$. Now fix $m$ and let $a_{s+1}$ be the largest number in $A$ such that $a_{s+1}\le m$. From the condition \eqref{condition23}, we must have either one (or more) of the following
\[ x_{s+1}\ge 1, x_s\ge 1, x_{s-1}\ge 2, x_{s-2}\ge 3, \ldots, x_{1}\ge s.  \]
Now the argument goes very similar with that of Theorem \ref{restricedpar}. We then obtain
\[ P(m,n,A)\le P(m-a_{s+1},n-1,A)+P(m-a_{s},n-1,A)+\cdots+P(m-sa_{1},n-s,A). \]
Next, applying our inductive hypothesis and Lemma \ref{lem} we have
\begin{align*}
P(m,n,A) &\le F(2n-3)+F(2n-3)+F(2n-5)+\cdots+F(2n-2s-1)\\
&\le F(2n-3)+F(2n-2)=F(2n-1),
\end{align*}
hence completing our inductive proof. 
\end{proof}

\begin{corollary}\label{p=2,3}
For $m, n\in\mathbb Z^+$, we have $P(m,n,A_q)\le F(2n-1)$ for $q\ge 2$.
\end{corollary}

\begin{proof}
From the last theorem, it suffices to show that $A_q$ satisfies the condition \eqref{condition23} with $q\ge 2$. Indeed, we claim that
for any $r\in\mathbb N$, \[ \sum_{i=1}^{r+1}i\cdot q^{r+1-i}<q^{r+2}. \]
Proceeding by induction on $r$, the base case $r=0$ is obviously true. Assume inductively that the inequality holds for some $r$. Now observe that
\[ \sum_{i=1}^{r+2}i\cdot q^{r+2-i}=\sum_{i=1}^{r+1}i\cdot q^{r+1-i}+\sum_{j=0}^{r+1}q^{r+1-j}\le\sum_{i=1}^{r+1}i\cdot q^{r+1-i}+q^{r+2}-1. \] 
By the inductive hypothesis, the last term is less than $q^{r+2}+q^{r+2}-1<q^{r+3}$; hence completing our inductive proof.
\end{proof}

\begin{remark}
Since $F(n) \le F(2n-1)$, Theorem \ref{A_q>3} immediately implies Corollary \ref{p=2,3} for $q > 3$. It can
also be observed that condition \eqref{condition23} is weaker than \eqref{condition*}, i.e., that any set $A$ satisfying
condition \eqref{condition*} also satisfies condition \eqref{condition23}. The weaker condition allows one to consider a larger
collection of sets $A$, particularly $A_2$ and $A_3$, but at the expense of a bound on $P(m, n,A)$ that is not as good as the original.
\end{remark}

We end this section with an interpretation of our results in terms of $q$-ary partitions, which are partitions of an integer into powers of $q$. It follows that $P(m,n,A_q)$ is the number of $q$-ary partitions of $m$ into at most $n$ parts. Hence, rephrasing the above results, we obtain the following

\begin{corollary}\label{q-ary}
For $m, n\in\mathbb Z^+$, the number of $q$-ary partitions of $m$ into at most $n$ parts is no more than $F(n)$ if $q\ge 4$ or $F(2n-1)$ if $q=2$ or $3$.
\end{corollary}

\section{Cohomology of $SL_2$}\label{cohomology}

The goal in this section is to bound the dimension of cohomology for Weyl modules over $SL_2$ using results from the previous section. For general background of rational cohomology of algebraic groups, the audience may refer to \cite{Jan}. We only introduce here necessary material for our calculations. We use the same notation and conventions as in \cite{LNZ}. In particular, let $k$ be an algebraically closed field of prime characteristic $p>0$. We fix $G=SL_2$ defined over $k$ and a torus subgroup $T$ of $G$. Then the set of dominant weights associated with $T$ can be identified with $\mathbb N$. Weyl modules (over $SL_2$) are indecomposable modules that are parametrized by dominant weights. Explicitly, we denote $V(m)$ the Weyl module of highest weight $m$ for each $m\in\mathbb N$ with $V(0)=k$ the trivial module.  

For $G$-modules $M$ and $N$, $\ext^n_{G}(M,N)$ is the $n$-th degree extension space of $M$ by $N$. When $M=k$, this space is called the $n$-th cohomology space of $G$ with coefficients in $N$ and denoted H$^n(G,N)$. The notation $\dim\opH^n(G,N)$ (or $\dim\ext^n_{G}(M,N)$) denotes the dimension of the cohomology (or extension) as a vector space over $k$. We are interested in estimating an upper bound for these quantities. 

We recall from \cite[Theorem 4.2]{LNZ} that if $p$ is an odd prime, then for any integers $m, n\ge 0$ the dimension $\dim\opH^n(SL_2,V(m))$ is equal to the number of solutions to the system
\begin{align}\label{system for total dim}
\begin{cases}
\displaystyle{2\sum_{i=1}^r a_i+\sum_{j=1}^rb_j=n+1},\\
\displaystyle{b_1+\sum_{i=1}^{r-1}(a_i+b_{i+1})p^i+a_rp^r=\frac{m}{2}+1,}
\end{cases}
\end{align}
where all $a_i$s are in $\mathbb N$, $b_i$ is either $0$ or $1$.\footnote{Note that there is an abuse of notation here. All the $a_i$s are now not elements of $A$ as in the previous part of the paper. Instead, these $a_i$s (and $b_i$s) are variables in this context.} Here $r$ is a sufficiently large integer in term of $m$. Note that the system \eqref{system for total dim} has no solutions if $m$ is odd. Therefore, whenever considering the cohomology $\opH^n(SL_2,V(m))$ we are only interested in the case when $m$ is even. 

Let $N(m,n)$ be the number of solutions to the system 
\begin{align}\label{system1}
\begin{cases}
\displaystyle{2\sum_{i=1}^r a_i+\sum_{j=1}^rb_j=n},\\
\displaystyle{b_1+\sum_{i=1}^{r-1}(a_i+b_{i+1})p^i+a_rp^r=m.}
\end{cases}
\end{align}
Then we can deduce that $\displaystyle\dim\opH^n({SL_2},V(m))=N\left(\frac{m}{2}+1,n+1\right)$ for $m,n\in\mathbb N$. We next prove the main result of this section, which strengthens \cite[Proposition 4.4]{LNZ} as we are now able to remove the condition $n\le 2p-3$. 

\begin{theorem}\label{casep>3}
Assume $p\ge 5$. For all integers $m, n\ge 0$, we have 
\[ \dim\opH^n(SL_2,V(m))\le F(n+1). \]
\end{theorem}

\begin{proof}
From earlier set up, we need to prove that $N\left(\frac{m}{2}+1,n+1\right)\le F(n+1)$ for every integer $n\ge 0$ and even integer $m\ge 0$. This is then reduced to showing that $N(m,n)\le F(n)$ for all positive integers $m, n$ (with $m$ replacing $\frac{m}{2}+1$ and $n$ replacing $n+1$). We again proceed by an inductive argument on $n$. By \cite[Proposition 3.6]{LNZ}, the last inequality holds for $n\le 8$. For any positive $n\ge 9$, assume that the inequality holds up to $n-1$. Since $b_1$ is either $0$ or $1$, we must have $m$ is congruent to either $0$ or $1$ modulo $p$, for otherwise, $N(m,n)=0$ for all $n$. If $m\equiv 1$ (mod $p$), then $b_1=1$ and $N(m,n)=N(m-1,n-1)\le F(n-1)$ by the inductive hypothesis. Suppose that $p$ divides $m$, so $b_1$ must be zero. Let $s$ be the least integer such that $p^s\le m$. Then the system \eqref{system1} is reduced to
\begin{align}\label{p|m}
\begin{cases}
\displaystyle{2\sum_{i=1}^{s} a_i+\sum_{j=2}^{s+1}b_j=n},\\
\displaystyle{\sum_{i=1}^{s}(a_i+b_{i+1})p^i=m.}
\end{cases}
\end{align}
Let $S$ be the set of solutions to this system. Set
\[ U=\left\{(a,b)=(a_1,\ldots,a_s,b_2\ldots,b_{s+1}): (a,b) \text{~satifies \eqref{p|m}}, ~\exists i_0 \text{~so that~} a_{i_0}\ge 1, b_{i_0+1}=0  \right\}.\]
Then $S$ is the union of the disjoint subsets $U$ and $V=S\setminus U$. We now give an upper bound to each set. For each solution $(a,b)$ in $U$, we choose the largest such $i_0$ and make a replacement $a_{i_0}\mapsto a_{i_0}-1$ and $b_{i_0+1}\mapsto 1$. The resulting tuple is a solution of 
\begin{align*}
\begin{cases}
\displaystyle{2\sum_{i=1}^{s} a_i+\sum_{j=2}^{s+1}b_j=n-1},\\
\displaystyle{\sum_{i=1}^{s}(a_i+b_{i+1})p^i=m.}
\end{cases}
\end{align*}
This replacement is a one-to-one mapping from $U$ to the set of solutions to the system above. It follows from the inductive hypothesis that $$|U|\le N(m,n-1)\le F(n-1).$$ 

Every solution $(a,b)$ in $V$ satisfies the condition that whenever $a_i>0$, $b_{i+1}=1$. Setting $d_i=a_i+b_{i+1}$ for all $1\le i\le s$, we can see that each solution $(a,b)$ in $V$ is one-to-one mapping to a solution $(d_1,\ldots,d_s)$ to the system
\begin{align}\label{d-system}
\begin{cases}
\displaystyle{\sum_{i=1}^{s} d_i\le n },\\
\displaystyle{\sum_{i=1}^{s}d_ip^i=m,}
\end{cases}
\end{align}
where the inequality is obtained from rewriting the first equation of \eqref{p|m} to $\displaystyle{\sum_{i=1}^{s} d_i=n-\sum_{i=1}^{s}a_i}$. Consider the following cases.

\begin{itemize}
\item If $\displaystyle \sum_{i=1}^{s}a_i\le 2$ then $d_i\le a_i+1\le 3<p$ for all $1\le i\le s$ (recall that $p\ge 5$). Hence, the sum $\displaystyle \sum_{i=1}^{s}d_ip^i$ must be the $p$-adic expansion of $m$ and so there is at most one solution to the system \eqref{d-system} for the case when $\displaystyle n-2\le \sum_{i=1}^{s} d_i\le n$.
\item If $\displaystyle \sum_{i=1}^{s}a_i\ge 3$ then $\displaystyle{\sum_{i=1}^{s} d_i=n-\sum_{i=1}^{s}a_i}\le n-3$. By Theorem \ref{A_q>3}, the number of solutions to \eqref{d-system}, restricted to this case, is bounded by $F(n-3)$.
\end{itemize}
Summing up the two cases, we have $|V|\le 1+F(n-3) \le F(n-2)$. Therefore, we obtain
\[ N(m,n)=|S|=|U|+|V|\le F(n-1)+F(n-2)=F(n),  \]
which completes our inductive proof.
\end{proof}

We believe that the theorem also hold for $p=3$. Unfortunately, our method does not work with this small prime. A different approach might be needed to tackle this case. Using the same idea as in the previous section, we can only prove the following

\begin{proposition}
If $p=3$, then we have for all integers $m\ge 0, n\ge 1$ 
\[ \dim\opH^n(SL_2,V(m))\le F(2n). \]
\end{proposition}

\begin{proof}
Same argument as in the proof of the last theorem, we reduce to showing that $N(m,n)\le F(2n-2)$ for all integers $m\ge 1, n\ge 2$. Recall that $N(m,n)$ is the number of solutions to the system \eqref{p|m}. Again, by \cite[Proposition 3.6]{LNZ}, the inequality holds for $n\le 4$. For $n\ge 5$, we assume that the inequality holds up to $n-1$. For \eqref{p|m} to have solutions, we must have either one (or more) of the following conditions
\begin{align*}
a_s+b_{s+1}\ge 1, a_{s-1}+b_s\ge 2, a_{s-2}+b_{s-1}\ge p+1, a_{s-3}+b_{s-2}\ge p+2,\\
a_{s-4}+b_{s-3}\ge p^2-p+1, a_{s-5}+b_{s-4}\ge p^2+1, a_{s-6}+b_{s-5}\ge p^3+1,
\end{align*}
and $a_{s-i}+b_{s-i+1}\ge p^{\lfloor\frac{i}{2}\rfloor}+1$ for all $i\ge 6$. For otherwise, we would have
\[ \sum_{i=1}^{s}(a_i+b_{i+1})p^i\le 2(p^{s-1}+p^{s-2}+\cdots+1)\le p^s-1<p^s\le m.  \]
Let $N_i$ be the number of solutions to the system \eqref{p|m} restricted to each of these conditions respectively. Then it is easy to see that $N(m,n)$ is no more than the sum of all these $N_i$s. We next consider each $N_i$.

If $a_{s}+b_{s+1}\ge 1$, then there are 2 cases:
\begin{itemize}
\item $b_{s+1}=0$ and $a_s\ge 1$. Then the system \eqref{p|m} can be rewritten to  
\begin{align*}
\begin{cases}
\displaystyle{2\sum_{i=1}^{s-1} a_i+\sum_{j=2}^{s}b_j=n-2a_s},\\
\displaystyle{\sum_{i=1}^{s-1}(a_i+b_{i+1})p^i=m-a_sp^{s}.}
\end{cases}
\end{align*}
Hence, by the inductive hypothesis the number of solutions in this case is
\[ \sum_{a_s=1}^{\lceil\frac{m}{p^s}\rceil}N(m-a_sp^s,n-2a_s)\le \sum_{a_s=1}^{\lceil\frac{m}{p^s}\rceil}F(2n-4a_s-2). \]

\item $b_{s+1}=1$ and $a_s\ge 0$. Then the system \eqref{p|m} can be rewritten to  
\begin{align*}
\begin{cases}
\displaystyle{2\sum_{i=1}^{s-1} a_i+\sum_{j=2}^{s}b_j=n-1-2a_s},\\
\displaystyle{\sum_{i=1}^{s-1}(a_i+b_{i+1})p^i=m-p^s-a_sp^{s}.}
\end{cases}
\end{align*}
Again, by the inductive hypothesis the number of solutions in this case is
\[ \sum_{a_s=0}^{\lceil\frac{m}{p^s}\rceil}N(m-p^s-a_sp^s,n-1-2a_s)\le \sum_{a_s=0}^{\lceil\frac{m}{p^s}\rceil}F(2n-4a_s-4). \]
\end{itemize}
Now summing up theses 2 cases and using Lemma \ref{lem}, we have
\[ N_1\le \sum_{i=2}^{\lceil\frac{m}{p^s}\rceil}F(2n-2i)\le F(2n-3).  \]

If $a_{s-1}+b_s\ge 2$ then $a_{s-1}\ge 1$. Now replacing $a_{s-1}$ by $a'_{s-1}+1$ with $a'_{s-1}\ge 0$ we have \eqref{p|m} rewritten to
\begin{align*}
\begin{cases}
\displaystyle{2\sum_{i=1}^{s-2} a_i+2a'_{s-1}+2a_s+\sum_{j=2}^{s+1}b_j=n-2},\\
\displaystyle{\sum_{i=1}^{s-2}(a_i+b_{i+1})p^i+(a'_{s-1}+b_{s})p^{s-1}+(a_s+b_{s+1})p^s=m-p^{s-1}.}
\end{cases}
\end{align*}
Hence, there are $N_2=N(m-p^{s-1},n-2)$ solutions in this case. Similar argument can be applied to obtain
\begin{itemize}
\item $N_3=N(m-p^{s-1}, n-2p)$
\item $N_4=N(m-p^{s-2}-p^{s-3}, n-2p-2)$
 $\cdots$ 
\item $N_{s}=N(m-p^{\lfloor\frac{s}{2}\rfloor},n-2p^{\lfloor\frac{s}{2}\rfloor})$.
\end{itemize}
Now using the inductive hypothesis and Lemma \ref{lem}, we obtain
\begin{align*}
N(m,n) &\le  F(2n-3)+F(2n-6)+F(2n-4p-2)+\cdots+F(2n-4p^{\lfloor\frac{s}{2}\rfloor}-2)\\
&\le F(2n-2),
\end{align*}
completing our inductive proof.
\end{proof}

\begin{remark}
Theorem \ref{casep>3} does not hold for $p=2$. Indeed, from \cite[Corollary 3.2.2]{EHP}, the dimension of $\opH^n(SL_2, V(m))$, for any $m, n\in\mathbb N$, is equal to the number of solutions $(a_1,\ldots,a_r)\in\mathbb N^r$ of the system
\begin{align}\label{p=2}
\begin{cases}
a_1+a_2+\cdots+a_r=n+1,\\
a_12+a_22^2+\cdots+a_r2^r=\frac{m}{2}+1.
\end{cases}
\end{align}
In the case when $m=286, n=4$, there are 6 solutions to the system \eqref{p=2} as follows
\begin{itemize}
\item $1\cdot 2^4+4\cdot 2^5=144$
\item $3\cdot 2^4+1\cdot 2^5+1\cdot 2^6=144$
\item $2\cdot 2^3+2\cdot 2^5+1\cdot 2^6=144$
\item $2\cdot 2^2+1\cdot 2^3+2\cdot 2^6=144$
\item $4\cdot 2^2+1\cdot 2^7=144$
\item $2\cdot 2^1+1\cdot 2^2+1\cdot 2^3+1\cdot 2^7=144$.
\end{itemize}
Hence, $\dim\opH^4(SL_2, V(286))=6>F(5)$. Instead, we can have the same bound as for the case $p=3$ as follows.
\end{remark}

\begin{corollary}
If $p=2$, then for all integers $m\ge 0, n\ge 1$ 
\[  \dim\opH^n(SL_2,V(m))\le F(2n). \]
\end{corollary}

\begin{proof}
By Corollary \ref{p=2,3}, it's straightforward to have
\[  \dim\opH^n(SL_2, V(m))\le P\left(\frac{m}{2}+1,n+1, A_2\right)\le F(2n+1). \]
In fact, Corollary \ref{p=2,3} implies that
\[  \sum_{i=0}^n\dim\opH^n(SL_2, V(m))\le F(2n+1). \]
It is possible to lower the bound for $\dim\opH^n(SL_2, V(m))$ to $F(2n)$. Indeed, let $M(\frac{m}{2}+1,n+1)$ be the number of solutions to the system \eqref{p=2}. Similar inductive argument as in the proof of Theorem \ref{A23} may be used to establish 
\[ M(m,n)\le F(2n-2) \] 
for all integers $m\ge 1, n\ge 2$, which is sufficient to show the corollary. 
\end{proof}

\begin{remark}
Using exact same proof for \cite[Theorem 5.4]{LNZ} with Theorem \ref{casep>3} replacing \cite[Proposition 4.4]{LNZ} (following that the condition $n\le 2p-3$ can be removed), we can prove that for $p\ge 5$ \[  \dim\ext^n_{SL_2}(V(m_2), V(m_1))\le F(n+1)+(s-1)F(n). \] 
for $m_1, m_2, n\in\mathbb N$, and $s$ the least positive integer such that $m_2<p^s$. This is not a significant bound as it is not sharp even for the low degree $n$. For example, we have 
\[ \dim\ext^n_{SL_2}(V(m_2), V(m_1))\le n  \]
for $n\le 3$, see \cite[Section 5.1]{LNZ} for details. Finding a sharp bound, for large values of $n$, of these extension spaces is still an open problem. We propose the following
\end{remark}

\begin{conjecture}
For $m_1, m_2\in\mathbb N$, and $p\ge 3$, we have
\[  \dim\ext^n_{SL_2}(V(m_2), V(m_1))\le F(n+1). \] 
\end{conjecture}

\begin{remark}
The same arguments as in \cite[Section 6.2]{LNZ} can show, in the case when $p\ge 5$, that both $\dim\opH^n(SL_2,L)$ and $\dim\opH^n(SL_2(\mathbb F_{p^s}),L')$ are bounded by $(2n+7)F(n)$, where $L$ (resp. $L'$) is any simple module over $SL_2$ (resp. the finite group of Lie type $SL_2(\mathbb F_{p^s})$ for any $s\ge 1$). Again, this is an improvement of results in \cite[Section 6.2]{LNZ}, but it is not a sharp upper bound even with small values of $n$.  

\end{remark}

\section*{acknowledgement}
We thank Sergey Fomin for valuable suggestions improving the manuscript. We are also grateful to the anonymous referee for pointing out errors in previous versions and providing many useful comments on the writing and organization of the paper. 


\providecommand{\bysame}{\leavevmode\hbox to3em{\hrulefill}\thinspace}

\end{document}